\documentclass[5p,review,twocolumn]{elsarticle}

\journal{Journal of \LaTeX\ Templates}

\usepackage{amsmath,amsfonts,calrsfs,amssymb,color,dsfont}
\usepackage{amsthm,enumerate}

\newcommand{\I}{\one}
\newcommand{\one}{\mathds 1}
\newcommand{\eps}{\varepsilon}

\renewcommand{\O}{\mathcal O}
\newcommand{\bO}{\partial\mathcal O}
\renewcommand{\L}{\mathcal L}
\def\R{\mathbb R}

\def\E{\mathbb E}
\def\P{\mathbb P}
\def\ds{\displaystyle}
\def\dd{\displaystyle}

\newtheorem{Theorem}{Theorem} [section]

\newtheorem{Proposition}[Theorem]{Proposition}

\newtheorem{Remark}[Theorem]{Remark}

\begin{document}

\begin{frontmatter}

\title{ Exact controllability  of stochastic differential equations with multiplicative noise}

\author{V. Barbu}
\address{Al.I. Cuza University and Romanian Academy,  Iasi, Romania}
\ead{vb41@uaic.ro}

\author{L. Tubaro}
\address{University of Trento, Italy}
\ead{luciano.tubaro@unitn.it}

\begin{abstract}
One proves that the $n$-D stochastic controlled equation $dX+A(t)Xdt=\sigma(X)dW+B(t)u\,dt$,
where $\sigma\in\mbox{Lip}((\R^n,\L(\R^d,\R^n))$,  $A(t)\in\L(\R^n)$ and $B(t)\in\L(\R^n,\R^n)$
is invertible, is exactly controllable with high probability in each $y\in\R^n$, $\sigma(y)=0$ on each finite interval $(0,T)$.
An application to approximate controllability to stochastic heat equation is given. The case where
$B\in\L(\R^m,\R^n)$, $1\le m <n$ and the pair $(A,B)$ satisfies the Kalman rank condition is also studied.
\end{abstract}

\begin{keyword}
stochastic equation\sep controllability\sep feedback controller
\MSC[2010] 60H10\sep  60H15 \sep 93B05\sep 93B52
\end{keyword}

\end{frontmatter}


\section{Introduction}
Consider the stochastic $n$-D differential equation
\begin{equation}\label{e1.1}
\begin{array}{l}
dX+A(t)\,X\,dt=\sigma(X)\,dW+B(t)u\,dt,\quad t\ge0\\
X(0)=x\in\R^n,
\end{array}
\end{equation}
where $\sigma\colon \R^n\to \L(\R^d,\R^n)$;   $A(t)\in\L(\R^n)$, $B(t)\in  \L(\R^m,\R^n)$,
$t\in [0,T]$, are assumed to satisfy
the following hypotheses
\begin{enumerate}[(i) ]
\item $y\in\R^n$, $\sigma\in \mbox{Lip}(\R^n,\L(\R^d,\R^n)),\quad \sigma(y)=0.$
\item $A,\,B\in C(\R^+;\L(\R^n,\R^n))$ and for some $\gamma>0$
\begin{equation}\label{e1.4}
B(t)\,B^*(t)\ge \gamma^2\,I,\ \ \forall t\in[0,\infty).
\end{equation}
\item $\sigma(X)\,dW(t)=\sum_{j=1}^d\sigma_{\cdot j}(X)\,d\beta_j(t)$,
	$t\ge 0$ where $\{\beta_j\}_{j=1}^d$ is a system of independent Brownian motions in the probability space
	$\{\Omega,\mathcal F,\P\}$.
\end{enumerate}
We denote by $(\mathcal F_t)_{t\ge0}$ the filtration corresponding to $\{\beta_j\}_{j=1}^d$ and by $X^u$ the solution to \eqref{e1.1}.

The problem we address here is the following
\paragraph{Problem 1} Given $x,y\in \R^n$ find an $(\mathcal F_t)_{t\ge0}$-adapted controller
$u\in L^2((0,T)\times\Omega;\R^m)$ such that
\begin{equation}\label{e1.3}
X^u(0)=x,\ \ X^u(T)=y.
\end{equation}

The main result of this work, Theorem 2.1 below, amounts to saying that, under
hypotheses (i)--(iii), Pro\-blem 1 has a solution $u^\ast$ in a sense to be made precised
later on and morover the controller $u^\ast$ can be found in a feedback form $u^\ast=\Phi^\ast(X)$.

As regards the literature on exact controllability of equation \eqref{e1.1} the works
\cite{EK}--\cite{WYYY} should be primarily cited. In par\-ti\-cular, in the recent work \cite{WYYY}
it is solved the above \mbox{exact} controllability problem in the special case where $\sigma$ is linear
and $B\equiv B(t)$ satisfies the condition \eqref{e1.4}.

With respect to above mentioned papers the main novelty of this work is the 
exact controllability of equation 
via a new controllability approach to \eqref{e1.1} by designing a feedback controller $u^\ast$ of relay type
which steers with high probability $x$ in $y$ in the time $T$. This constructive approach allowed to solve the controllability problem
for control systems \eqref{e1.1} with Lipschitzian volatility term $\sigma$.

\section{The main result}
\begin{Theorem}\label{t2.1}
	Assume that hypotheses {\rm(i)--(iii)} hold. Let $x,y\in \R^n$ and $T>0$ be arbitrary
	but fixed. Then, for each $\rho>0$, there is an $(\mathcal F_t)_{t\ge0}$-adapted
	controller $u^\ast\in L^\infty((0,T)\times\Omega;\R^m)$ such that if
	\begin{equation}\label{e2.1}
	\tau=\inf\{t\ge 0: |X^{u^\ast}(t)-y|=0 \},
	\end{equation}
	we have
	\begin{equation}\label{e2.2}
	\P(\tau\le T)\ge 1-(\rho\eta)^{-1}\,(|y|+(1-e^{-C^\ast\,T})^{-1}\,|x-y|)
	\end{equation}
	for some $\eta, C^\ast>0$ independent of $\rho, \;x$ and $y$. Moreover, the controller $u^\ast$ is
	expressed in the feedback form
	\begin{equation}\label{e2.3}
	u^\ast(t)\in -\rho\,\mbox{\em sign}(B^\ast(t)(X(t)-y)),\quad t\in(0,T).
	\end{equation}
\end{Theorem}
Here $\mbox{sign}\colon \R^n\to \R^n$ is the multivalued mapping
\begin{equation}\label{e2.4}
\mbox{sign}\,y=
\begin{cases}
\frac{y}{|y|}& \mbox{if }y\not=0\\
\{\theta\in\R^n : |\theta|\le 1\} & \mbox{if }y=0.
\end{cases}
\end{equation}
In a few words the idea of the proof is to show that the corresponding closed loop stochastic system
\begin{equation}\label{e2.5}
\begin{array}{l}
\begin{aligned}
dX(t)+A(t)\,X(t)\,dt+\rho\,B(t)\,\mbox{sign}(B^\ast(t)(X(t)-y))\,dt\ni\\ \sigma(X)\,dW,
\end{aligned}\\
X(0)=x
\end{array}
\end{equation}
is well posed that is, it has a unique absolutely continuous solution, and that if $\tau$ is the stopping time defined by \eqref{e2.1} then \eqref{e2.2} holds.
By \eqref{e2.3}-\eqref{e2.4} we see that $u^\ast$ is a relay controller given by
\[
\begin{cases}
u^\ast(t)\!=\!-\rho\,U(X(t))\,|U(X(t))|&\!\mbox{on }\{(t,\omega)\mid U(X(t))\not=0\}\\
|u^\ast(t)|\le \rho&\!\mbox{on }\{(t,\omega)\mid U(X(t))=0\}
\end{cases}
\]
where $U(X(t))= B^\ast(t)(X(t)-y)$.
Though $u^\ast$ is not explicitely defined on $G=\{(t,\omega)\mid U(X(t))=0\}$, it is however an $\mathcal F_t$-adapted
controller multivalued process which is uniquely defined on $G^{\mbox{\it c}}$, i.e. the complement of $G$.

\bigskip
Theorem \ref{t2.1} amounts to saying that under assumptions
(i)--(iii), system \eqref{e1.1} is exactly controllable to each $y\in\sigma^{-1}(0)$
with high probability for $\rho$ large enough. In particular one has exact null controllability if $\sigma(0)=0$.

We shall denote by the same symbol $|\cdot|$ the norm in the Euclidean spaces $\R^n$ and
$\L(\R^n,\R^m)\; = \;\R^{nm}$. For $n=m$ we simply write $\L(\R^n,\R^n)=\L(\R^n)$.

\section{Proof of  Theorem \ref{t2.1}}
We have
\begin{Proposition}\label{p3.1}
	Let $0<T<\infty$. There is a unique strong solution $X\in L^2(\Omega;C([0,T];\mathcal L(\R^{n})))$
	to \eqref{e2.5}. More precisely, there are $X\in L^2(\Omega;C([0,T];\mathcal L(\R^{n})))$ and an
	$(\mathcal F_t\}_{t\ge0}$-adapted process $\xi\colon [0,T]\to \mathcal L(\R^{n})$ such that
	$\xi\in L^\infty((0,T)\times\Omega;\mathcal L(\R^{n}))$ and
\begin{eqnarray}\label{e3.6}
&\hspace*{-10mm}\xi(t)\!\in\!  B(t)(\mbox{\em sign} \,(B^\ast(t)(X(t)-y))),\mbox{ a.e. in }(0,T)\!\times\!\Omega
	\\[2mm]
	\label{e3.7}
&\begin{array}{l}
	dX(t)+A(t)\,X(t)\,dt+\rho\,\xi(t)\,dt=\sigma(X)\,dW\\
	X(0)=x
\end{array}
\end{eqnarray}
\end{Proposition}
We shall prove Proposition \ref{p3.1} at the end of this section and now we use it to prove  Theorem \ref{t2.1}.
The proof is based on some extinction type arguments already deve\-loped in a different context in \cite{BBT} and \cite[pag. 68]{BDR}.
(In the following we shall write $A$ instead of $A(t)$.)

\bigskip
Let $\varphi_\eps\in C^2(\R^+)$ be such that $\varphi_\eps(r)=\frac{r}{\eps}$
for $0\le r\le \eps$, $\varphi_\eps'(r)=1+\eps$ for $r\ge 2\eps$ and
$|\varphi_\eps''(r)\le\frac{C}{\eps}$, $\forall r\in\R^+$. We set $\Phi_\eps(X)=\varphi_\eps(|X|)$, $\forall X\in\R^n$.
We have
$\nabla\Phi_\eps(X)=\varphi_\eps'(|X|)\,\mbox{sign}\,X$, $\nabla^2\Phi_\eps(X)=0$
for $|X|\le\eps$ and $|X|\ge2\eps$, $|\nabla^2\Phi_\eps(X)|\le \frac{C}{\eps}$.
We apply It\^o's formula in \eqref{e3.7} to function $t\to \Phi_\eps(X(t)-y)$.
We get
\[
\begin{array}{l}
d\Phi_\eps(X(t)-y)\\+\langle A(X(t))-A(y),\nabla\Phi_\eps(X(t)-y)\rangle\,dt\\
+\rho\,\langle\xi(t),B^\ast(t)\nabla\Phi_\eps(X(t)-y)\rangle\,dt=\\
-\langle A(y),\nabla\Phi_\eps(X(t)-y)\rangle\,dt\\+\tfrac12\sum_{i,j=1}^n\alpha_{ij}(\nabla^2\Phi_\eps(X(t)-y))_{ij}\,dt\\
+\langle \sigma(X(t))\,dW,\nabla\Phi_\eps(X(t)-y)\rangle
\end{array}
\]
where $\alpha_{ij}=\sum_{\ell=1}^d\sigma_{i\ell}\sigma_{\ell j}$. We note that for $\eps\to0$
$\Phi_\eps(X(t)-y)\to |X(t)-y|$, $\nabla\Phi_\eps(X(t)-y)\to \eta(t)\in\mbox{sign}\,(X(t)-y)$,
$|\nabla\Phi_\eps(X(t)-y)|\le 1+\eps$, and because $\sigma(y)=0$ we have also
$|\alpha_{ij}(t)(\nabla^2\Phi_\eps(X(t)-y))_{ij}|\le C_2^\ast\,\eps$ for all $t\ge0$.

\bigskip
On the other hand, by \eqref{e1.4} it follows that there is $\gamma>0$ such that
\begin{equation}\label{e3.8}
|B^\ast(t)(X(t)-y)|\ge\gamma\,|X(t)-y|
\end{equation}
We note also that 
$$|\alpha_{ij}(t)|\le  C_2\,|X(t)-y|^2.$$ 
Integrating on $(s,t)\subset(0,\infty)$ we get
\[
\begin{array}{l}
\ds\Phi_\eps(X(t)-y)+\rho\int_s^t\langle\xi(r),B^\ast(t)\nabla\Phi_\eps(X(r)-y)\rangle\,dr\le\\\ds
\Phi_\eps(X(s)-y)+\|A\|(1+\eps)\int_s^t(|y|+|X(r)-y|)dr\\\ds+C_2^\ast\eps+\int_s^t\langle \sigma(X(r))\,dW,\nabla\Phi_\eps(X(r)-y)\rangle.
\end{array}
\]
Taking into account that $$B^\ast(t)\,\nabla \Phi_\eps(X(r)-y)\to B^\ast(t)\eta(r),$$
with $\eta(r)\in\mbox{sign}\,(X(r)-y)$ and that
\[
\langle \xi(r),B^\ast(t)\eta(r)\rangle=|B^\ast(t)(X(r)-y)|\,|X(r)-y|^{-1}\I_{|X(r)-y|\not=0},
\]
by \eqref{e3.8} we get for $\eps\to0$
\[
\begin{aligned}
&|X(t)-y|+\rho\,\gamma\int_s^t\one_{|X(r)-y|\not=0}\,dr\\
\le\,&|X(s)-y|
+C^\ast\int_s^t |X(r)-y|\,dr+C^\ast(t-s)\,|y|
 \\ &\hspace{2cm}\!\!+\int_s^t\langle \sigma(X(r))dW_r,\mbox{sign}(X(r)-y))\rangle
\end{aligned}
\]
where $C^\ast$ is independent of $x$, $y$ and $\rho$.
Hence
\begin{eqnarray}
&&\hspace*{-15mm}
e^{-C^\ast t}|X(t)-y|+ \rho\,\gamma\dd\int_s^te^{-C^\ast r}\one_{|X(r)-y|\not=0}\,dr\nonumber\\
&&\hspace*{-15mm}\le e^{-C^\ast s}|X(s)-y|+(1-e^{C^\ast(s-t)})|y| \label{e3.9}\\
&&\hspace*{-15mm} +\dd\int_s^te^{-C^\ast r} \langle \sigma(X(r))dW_r,\mbox{sign}(X(r)-y))\rangle, \ 0< s\le t<\infty\nonumber
\end{eqnarray}
In particular, \eqref{e3.9} implies that the process
$$t\to e^{-C^\ast t}|X(t)-y|$$ is a $(\mathcal F_t)_{t\ge0}$-supermartingale
that is,
\[
\E(e^{-C^\ast\,t}|X(t)-y|\mid \mathcal F_s)\le e^{-C^\ast\,s}|X(s)-y|,\ \ \forall t\ge s.
\]
This yields
$
|X(t)-y|=0,\ \forall t\ge \tau,
$
where $\tau$ is defined by \eqref{e2.1}.

If take expectation $\E$ in \eqref{e3.9}, we obtain, for $s=0$,
\[\begin{array}{c}
e^{-C^\ast t}\E|X(t)-y|+\rho\,\gamma\dd\int_0^te^{-C^\ast r}\P(\tau>r)\,dr\\
\le
|x-y|+(1-e^{-C^\ast t)})|y|.\end{array}
\]
Hence, for $t=T$ we get
\begin{equation}\label{e3.10}
\begin{array}{l}
\P(\tau>T)\\
\dd\le \frac{C^\ast}{\rho\,\gamma}\,
\big((1-e^{-C^\ast T})^{-1}|x-y|+|y|\big)\end{array}
\end{equation}

\bigskip\noindent

\textit{Proof of Proposition {\rm\ref{p3.1}}.}

Let $F_\lambda(t)\in C([0,T];\R^n)$ 
be~the Yosida approximation of $F(t,X)=\rho\,B(t)\mbox{sign}(B^\ast(t)( X-y))$,
 that is (see \cite[pag. 97]{B})
\begin{equation}\label{e3.12}
F_\lambda(t)=\tfrac1\lambda\,(I-(I+\lambda\,F(t))^{-1}),\quad \lambda>0
\end{equation}
We note that the operator $F(t)$ is $m$-accretive in  the space $\R^n\times\R^n$.
Since the $F_\lambda(t)$ are Lipschitz for $t\in[0,T]$, the equation
\begin{equation}
\begin{array}{l}
dX_\lambda+A(t)\,X_\lambda\,dt+F_\lambda(t,X_\lambda)\,dt=\sigma(X_\lambda)\,dW\\
X_\lambda(0)=x
\end{array}
\end{equation}
has for each $T>0$ a unique solution $$X_\lambda\in L^2(\Omega;C([0,T];\R^n).$$
Taking into account  that for each $\lambda>0$
\begin{eqnarray}\label{e3.14}
& F_\lambda(t,X)\in F(t,I+\lambda F(t))^{-1}X),\  \forall X\in \R^n
\\[1mm]
&|F_\lambda(t,X)| \le C\,\rho,\  \forall X\in \R^n,\;\lambda>0\label{e3.15}
\end{eqnarray}
and that $X\to F(t,X)$ is monotone in $\R^n$,
we get, via the Burkholder-Gundy-Davis inequality, the estimate
\[
\E\sup_{t\in[0,T]} |X_\lambda(t)|_{\mathcal L(\R^n)}^2\le C,\  \forall \lambda>0
\]
\[
\begin{array}{l}
\dd\E\sup_{t\in[0,T]} |X_\lambda(t)-X_\mu(t)|\\\qquad
\le C\,\E\dd\int_0^t (\lambda\,|F_\lambda(r,X_\lambda(r))|^2+\mu\,|F_\mu(r,X_\mu(r))|^2)\,dr\\\qquad
\le C\,(\lambda+\mu),\ \ \forall \lambda,\mu>0.
\end{array}
\]
Hence, there is
\begin{equation}\label{e3.16}
X=\lim_{\lambda\to0} X_\lambda\ \ \mbox{ in }\
L^2(\Omega;C([0,T];\R^n))
\end{equation}
and by \eqref{e3.14}, \eqref{e3.15} there is also (on a subsequence)
\begin{equation}\label{e3.17}
\xi=\mbox{w$^*$-}\lim_{\lambda\to0}F_\lambda(t,X_\lambda)
\ \ \mbox{ in }\  L^\infty((0,T)\times\Omega;\R^n)
\end{equation}
Since by \eqref{e3.12} and \eqref{e3.16}
\[
(I+\lambda F(t))^{-1}X_\lambda(t)\to X(t)\  \ \mbox{ in }\
L^2(\Omega;C([0,T]; \R^n)),
\]
for $\lambda\to 0$, it follows by \eqref{e3.14}, \eqref{e3.15} and the maximal monotonicity
of $F(t)\colon  \R^n\to R^n$
that
\[
\xi(t)\in F(t,X(t)),\ \ \mbox{a.e. in }(0,T)\times\Omega.
\]
Hence $X$ is a solution to \eqref{e3.6}-\eqref{e3.7} as claimed. The uniqueness
is immediate by monotonicity of the mapping $F(t)$ but we omit the details.

\section{ The case of linear multiplicative noise}
Consider here the equation
\begin{equation}\label{e55.1}
\begin{array}{l}
dX+A(t)\,X\,dt=\sum_{i=1}^d\sigma_i(X)\,d\beta_i+B(t)\,u(t)\,dt\\
X(0)=x
\end{array}
\end{equation}
with the final target $X(T)=y$,
where $B(t)$ satisfies assumption (ii) and $\sigma_i\in\L(\R^n)$.

Let $\Gamma\in C([0,T];\L(\R^n))$ be the solution to equation
\begin{equation}
d\Gamma(t)=\sum_{i=1}^d \sigma_i\Gamma(t)\,d\beta_i,\quad t\ge0,\ \ \Gamma(0)=I.
\end{equation}
By the substitution $X(t)=\Gamma(t)y(t)$ one transforms via It\^o's formula equation \eqref{e55.1}
into random differential equation
\begin{equation}\label{e55.4}
\frac{dy}{dt}(t)+\Gamma^{-1}(t)A(t)\Gamma(t)\,y(t)=\Gamma^{-1}(t)B(t)u(t).
\end{equation}
In \eqref{e55.4} we take $u$ the feedback controller
\begin{equation}\label{e55.5}
u(t)=-\tilde\rho\,\mbox{sign}\,\big((B(t)\Gamma^{-1}(t))^\ast(y(t)-y_T)\big),\ \ t\ge0
\end{equation}
where $y_T=\Gamma^{-1}(T)X_T$. Arguing as in the proof of Proposition \ref{p3.1}
it follows that \eqref{e55.4} has (for each $\omega\in\Omega$) unique absolutely continuous solution $y$
with $\frac{dy}{dt}\in L^2(0,T;\R^n)$. 
We note that if $y$ is an $(\mathcal F_t)_{t\ge0}$-adapted solution to \eqref{e55.4}-\eqref{e55.5} then $X=\Gamma(t)y(t)$
is the solution to closed loop system \eqref{e55.1} with feedback control
\begin{equation}\label{e55.5a}
\begin{split}
&u(t)=\\&-\tilde\rho\,\mbox{sign}\,\big((B(t)\Gamma^{-1}(t))^\ast\,\Gamma^{-1}(t)(X(t)-\Gamma(t)\Gamma^{-1}(T)X_T)\big).
\end{split}
\end{equation}

We have
\begin{Theorem}\label{t5.1}
Let $T>0$, $x\in\R^n$ and $X_T\in\mathcal F_T\cap L^2(\Omega)$
be arbitrary but fixed. Then there is $\tilde\rho\in\mathcal F_T\cap L^2(\Omega)$ such that
the feedback controller \eqref{e55.5}
steers $x$ in $y_T$, in time $T$, with probability one.
\end{Theorem}
\begin{proof}
If multiply equations \eqref{e55.4}-\eqref{e55.5} by $y(t)-y_T$ we get by \eqref{e1.4} that
\begin{equation}\label{e55.6}
\begin{array}{l}
\tfrac12\frac{d\ }{dt}|y(t)-y_T|^2+\tilde\rho\,\gamma\,C_1^\ast\, |y(t)-y_T|\le\\
C_2^\ast\,(|y(t)-y_T|+|y_T|)\,|y(t)-y_T|,
\end{array}
\end{equation}
a.e. $t\in(0,T)$, where
\[
\begin{array}{l}
(C_1^\ast)^{-1}=\sup\{\|(\Gamma^\ast(t))^{-1}\|_{\L(\R^n)};t\in[0,T]\},\\
C_2^\ast=\sup\{\|\Gamma^{-1}(t)A(t)\Gamma(t)\|_{\L(\R^n)};t\in[0,T]\}.
\end{array}\]
By \eqref{e55.6} it follows that if $\tilde\rho\,\gamma\,C_1^\ast>C_2^\ast|y_T|$ then the function
\[
t\to e^{-C_2^\ast\,t}|y(t)-y_T|+(\tilde\rho\,\gamma\,C_1^\ast-C_2^\ast|y_T|)(C_2^\ast)^{-1}(1-e^{-C_2^\ast\,t})
\]
is monotonically decreasing and so $y(T)-y_T=0$ if $\tilde\rho$ is taken in such a way that
\[
(\tilde\rho\,\gamma\, C_1^\ast-C_2^\ast|y_T|)(C_2^\ast)^{-1}(1-e^{-C_2^\ast\,T})\ge |x-y_T|.
\]
Then Theorem \ref{t5.1} follows for
\[
\tilde\rho=(\gamma\,C_1^\ast)^{-1}(C_2^\ast|y_T|+|x-y_T|).
\]
\end{proof}
It should be noted that since $\tilde\rho$ is not $\mathcal F_0$-measurable, the solution $y$ to system
\eqref{e55.4}-\eqref{e55.5} is not $(\mathcal F_t)_{t\ge0}$-adapted and so it is not equivalent with \eqref{e55.1}-\eqref{e55.4}.
This happens however if $A(t)$ and $B(t)$ commute with $\sigma_i$ because in this case $C_i^\ast$, $i=1,2$
are deterministic and so can be chose $\tilde\rho$.
In general it follows for system  \eqref{e55.1}-\eqref{e55.4} with $\tilde\rho=\rho$ and $y_T$ deterministic, a result similar to
that in Theorem \ref{t2.1}. Namely, by \eqref{e55.6} it follows as above (see \eqref{e3.9}) that
\[
\begin{split}
\E\big(e^{-C_2^\ast t}|y(t)-y_T|\big)+\rho\,\gamma\,\E\int_0^t e^{-C_2^\ast r} \P(\tau>r)\,dr\le\\ \E\big(|x-y_T|+(1-e^{-C_2^\ast T})|y_T|\big)
\end{split}
\]
and therefore $\P(\tau>T)\le 1-(\rho\eta)^{-1}(|x-y_T|+|y_T|)$ for some $\eta>0$.

\begin{Remark}\em
Clearly Theorem \ref{t5.1} extends to Lipschitzian mappings $A(t)\colon \R^n\to\R^n$.
\end{Remark}

Consider now system \eqref{e1.1} where $A\in\mathcal L(\R^n)$, $B\in\mathcal L(\R^m,\R^n)$,
$1\le m<n$ is time dependent and is satisfied the Kalman rank condition
\begin{equation}\label{e4.26a}
\mbox{rank}\|B,AB,\ldots,A^{n-1}B\|=n
\end{equation}
Assume also that $d=1$, $\sigma_1=\sigma$ and
\begin{equation}\label{e4.27a}
\sigma^k= a\,\sigma ,\quad \forall k\ge 2
\end{equation}
\begin{equation}\label{e4.28a}
\sigma(\R^d)\subset B(\R^m)
\end{equation}
for some $a\in \R$.\\
We have

\begin{Theorem}\label{t4.2}
Let $T>0$ and $x\in \R^n$ be arbitrary but fixed. Then under hypotheses
\eqref{e4.27a}-\eqref{e4.29a} there is an $(\mathcal F_t)_{t\ge 0}$-adapted controller $u\in L^2((0,T)\times\Omega;\R^n)$
which steers $x$ in origin, in time $T$, with probability one.
\end{Theorem}
\begin{proof}
By the transformation $X(t)=\Gamma(t)y(t)$ one reduces \eqref{e1.1} to the random system
\eqref{e55.4}, that is
\begin{equation}\label{e4.29a}
\begin{array}{l}
\begin{aligned}\frac{dy}{dt}(t)+\exp(-\beta(t)\,\sigma+\tfrac{t}2\,\sigma^2) A \exp(\beta(t)\,\sigma-\tfrac{t}2\,\sigma^2)=\\
\exp(-\beta(t)\,\sigma+\tfrac{t}2\,\sigma^2)B\,u(t)\end{aligned}\\
y(0)=x
\end{array}
\end{equation}
because $\Gamma(t)=\exp(\beta(t)\,\sigma-\tfrac{t}2\,\sigma^2)$.\\
Taking into account hypothesis \eqref{e4.27a}, we can rewrite \eqref{e4.29a} as
\begin{equation}\label{e4.30a}
\begin{array}{l}
\ds \frac{dy}{dt}+A\,y=Bu-\sigma\,D(t)\,y+\sigma\,D_1(t)\,u\\
y(0)=x.
\end{array}
\end{equation}
where $D_1(t)=\sum_{k=1}^\infty \frac{1}{k!}(\beta(t)-\frac{t\,a}2)^k\sigma^{k-1}$ and
\[
D(t)=\sum_{k=1}^\infty \frac{1}{k!}(-\beta(t)+\frac{t\,a}2)^k\sigma^{k-1}A\sum_{k=1}^\infty \frac{1}{k!}(\beta(t)-\frac{t\,a}2)^k\sigma^{k}
\]
Now by Kalman's condition \eqref{e4.26a} we know that there is a deterministic controller $\tilde u\in L^2(0,T;\R^m)$ such that
\begin{equation*}\label{}
\begin{array}{l}
\ds \frac{d\tilde y}{dt}+A\,\tilde y=B\tilde u,\quad t\in (0,T)\\
\tilde y(0)=x,\quad \tilde y(T)=0.
\end{array}
\end{equation*}
Since $B^{-1}\in \mathcal l(B(\R^n),\R^m)$, it follows by \eqref{e4.29a} that
\begin{equation*}\label{}
\begin{array}{l}
\begin{aligned}\ds \frac{d\tilde y}{dt}+\Gamma^{-1}A\,\Gamma\tilde y=\Gamma^{-1}B\tilde u(t)+\sigma\,D(t)\tilde y(t)+\sigma D_1(t)\tilde u\\
=\Gamma^{-1}B(\tilde u(t)+B^{-1}(\sigma D(t)\tilde y(t)+\sigma D_1(t)\tilde u))\\=\Gamma^{-1}B\tilde u(t)\end{aligned}\\
\tilde y(0)=x,\quad \tilde y(T)=0.
\end{array}
\end{equation*}
This means that $\big(\tilde y,u=\tilde u(t)+B^{-1}(\sigma D(t)\tilde y(t)+ D_1(t)\tilde u)\big)$
satisfies system \eqref{e4.30a} and $\tilde y(T)=0$. The controller $u$ is obviously $(\mathcal F_t)_{t\ge0}$-adapted
and so $\big(X(t)=\Gamma^{-1}(t)\tilde y(t),u(t)\big)$ satisfies system \eqref{e1.1} and $X(T)=0$ $\P$-a.s.
\end{proof}
\begin{Remark}\em
One might suspect that the controller $u$ steering $x$ in origin can be found in feedback form
but this problem is open.
\end{Remark}
%
\color{black}
\section{An example}
Consider the controlled $n$-order stochastic differential equation
\begin{equation}\label{e4.1}
\begin{array}{l}
X^{(n)}(t)+\dd\sum_{i=1}^{n}a_i\,X^{(i-1)}(t)\\\qquad\quad\ =\sigma_0(X,X',\ldots,X^{(n-1)})\,\dot W+u(t)
\\
\{X^{(k)}(0)\}_{k=0}^{n-1}=x\in\R^n
\end{array}
\end{equation}
where $a_i\in\R$, $\sigma_0(x_1,x_2,\ldots,x_n)=\sum_{i=1}^n b_i\,x_i$, $b_i\in\R$
and $W$ is a Wiener process in 1-D.

A typical example is the stochastic harmonic oscillator
$$\begin{array}{l}
\ddot X+a\dot X\,dt+bX\,dt=\sigma_0\,\dot W\\
X(0)=X_0,\ \dot X(0)=X_1.\end{array} $$
Equation \eqref{e4.1} is viewed as the stochastic differential system
\[
dX+A\,X\,dt=B\,u\,dt+\sigma(X)\,dW
\]
where $X=(X_i)_{i=1}^n$, $X_i=X^{(i-1)}$, $X(0)=x$,
\[
\sigma=\begin{pmatrix}0&0&\ldots&0\\ \vdots&\vdots&\vdots&\vdots\\0&0&\ldots&0\\b_1&b_2&\ldots&b_n\end{pmatrix}
\]
and
$$
A=
\begin{pmatrix}
0&1&0&\cdots&0\\0&0&1&\cdots&0\\
\vdots\\
0&0&0&\cdots&1\\
a_1&a_2&a_3&\cdots&a_{n}
\end{pmatrix},\ \ B=\begin{pmatrix}0\\0\\\vdots\\0\\1\end{pmatrix}.
$$
Clearly assumptions \eqref{e4.26a}-\eqref{e4.28a} hold and so by Theorem \ref{t4.2} it follows that, for each $x \in\R^n$,
 there is an $(\mathcal F_t)_{t\ge0}$-adapted feedback
controller $u^*(t)$ such that $X^{(i-1)}(T)=0$ for $i=1,2,\ldots,n-1$.

\section{Approximate controllability of stochastic heat equation}

Consider the stochastic equation
\begin{equation}\label{e5.1}
\begin{array}{l}
dX-\Delta\,X\,dt=\sum_{j=1}^d X\,e_j\,d\beta_j+\one_{\O_0}\,u\,dt,\\\hfill (t,\xi)\in(0,T)\times\O\\
X(0,\xi)=x(\xi),\  \xi\in\O\\
X(t,\xi)=0,\  \forall (t,\xi)\in(0,T)\times\bO.
\end{array}
\end{equation}
Here $d\ge 1$, $\O\subset \R^n$ is a bounded and open domain with smooth boundary
$\bO$, $\O_0$ is an open subset of $\O$ and $\{e_j\}_{j=1}^\infty$ is an orthonormal
base in $L^2(\O)$, given by $-\Delta\,e_j=\lambda_j\,e_j$ in $\O$, $e_j=0$ on $\bO$.
The controller $u\colon (0,\infty)\to L^2(\O)$ is an $(\mathcal F_t)_{t\ge0}$-adapted process.

We set $\tilde X^N=\sum_{i=1}^NX_i^N\,e_i$, $\tilde u^N=\sum_{i=1}^Nu_i^N\,e_i$
and approximate  \eqref{e5.1} by the $N$-D differential equation
\begin{equation}\label{e5.2}
\begin{array}{l}
dX^N-A_N\,X^N\,dt=\dd\sum_{j=1}^d \sigma_j(X^N)\,d\beta_j+B_N\,u^N\,dt,\vspace*{-3mm}\\ \\
X^N(0)=0
\end{array}
\end{equation}
where  
$$\begin{tabular}{l}
$X^N=\{X_i^N\}_{i=1}^N$, $u^N=\{u_i^N\}_{i=1}^N$,\vspace*{1mm}\\
$B_N=\left(\dd\int_{\O_0}e_i\,e_j\,d\xi\right)_{i,j=1}^N$,\\
 $A_N=\mbox{diag}(\lambda_i)_{i=1}^N$,
$\sigma_j(X^N)=\left(\dd\sum_{k=1}^n\langle e_k\,e_j,e_i\rangle_2\,X_k^N\right)_{i=1}^N$,
\end{tabular}$$
 \ \ $\langle \cdot,\cdot\rangle_2$ is the scalar product in $L^2(\O)$.\\
By the unique continuation property of eigenfunctions $e_j$, it follows that $\det B_N\ne0$,
which implies \eqref{e1.3}. Then, by Theorem \ref{t2.1}, for each $N\in\mathbb{N}$, equation \eqref{e5.1} is exactly controllable on $[0,T]$
in the sense of \mbox{\eqref{e2.1}--\eqref{e2.2}.} Taking into account that, 
\[
|x-\sum_{i=1}^N\langle x,e_i\rangle_2\,e_i|_2\to 0\  \mbox{ as }N\to\infty,
\]
\[
\E(\sup\{|\tilde X^N(t)-X^{\tilde u^N}(t)|_2^2,\;\;t\in[0,T]\})\to 0\ \mbox {as }N\to\infty,
\]
we get the following controllability result
\begin{Theorem}\label{t6.1}
	Let $x\in L^2(\O)$ and $T>0$ be arbitrary but fixed. Then for each $\eps>0$
	there is an $(\mathcal F_t)_{t\ge0}$-adapted controller
	$u_\eps\in L^2((0,T)\times\Omega;L^2(\O))$ such that
	\begin{equation}
	\P(|X^{u_\eps}(t)|_2\le\eps,\;\;\forall t \ge T)\ge 1-\eps.
	\end{equation}
\end{Theorem}

\begin{Remark}\rm In 1-D a similar result was established by a different method in \cite{6a}. It turns out (see \cite{3a}) that, under the above assumptions, there is an $(\mathcal{F}_t)_{t\ge0}$-adapted controller $u$ which steers $x$ into a linear subspace of $L^2(\Omega;\mathcal{O})$.\break However, it remains an open problem the exact null controllability. (For other partial results to exact null controllability, see \cite{6aa}, \cite{WYYY}.) 
	\end{Remark}

\section{Conclusion}
Under hypotheses (i)--(iii), the stochastic differential equation \eqref{e1.1} is  exactly controllable to any $y\in\sigma^{-1}(0)$ by a stochastic feedback controller
$u$ which is explicitly designed. In the special case of stochastic equations with
linear multiplicative noise the controllability set $\{y=X^u(T)\}$ is all $\R^n$. Moreover if the pair $(A,B)$ satisfies
the Kalman rank condition and $\sigma(\R^d)\subset B(\R^m)$ then the system \eqref{e1.1} is 
exactly null controllable.
As application the approximate controllability of stochastic heat equation with multiplicative Wiener noise was given.

\section*{Acknowledgments}
The authors thank the Mathematics Department of University of Trento for the financial support. V. Barbu
was supported by the grant of Romanian Ministry of Research and Innovation
CNCS-UEFISCDI, DN-III-D4-DCE-2016-0011

\end{document}